\documentclass[a4paper,12pt]{article}%
\usepackage{amssymb}
\usepackage{amsmath}
\usepackage{amsfonts}
\usepackage{graphicx}%
\providecommand{\U}[1]{\protect\rule{.1in}{.1in}}
\providecommand{\U}[1]{\protect\rule{.1in}{.1in}}
\newtheorem{theorem}{Theorem}

\newtheorem{corollary}[theorem]{Corollary}

\newtheorem{definition}[theorem]{Definition}

\newtheorem{lemma}[theorem]{Lemma}

\newtheorem{remark}[theorem]{Remark}

\newenvironment{proof}[1][Proof]{\noindent\textbf{#1.} }{\ \rule{0.5em}{0.5em}}
\begin{document}

\title{\textbf{Extrapolation and Local Acceleration of an Iterative Process }\\\textbf{for Common Fixed Point Problems}}
\author{Andrzej Cegielski$^{1}$ and Yair Censor$^{2}\medskip$\\$^{1}$Faculty of Mathematics, Computer Science and \\Econometrics, University of Zielona G\'{o}ra, \\ul. Szafrana 4a, 65-514 Zielona G\'{o}ra, Poland\\(a.cegielski@wmie.uz.zgora.pl)\medskip\\$^{2}$Department of Mathematics, University of Haifa\\Mt. Carmel, Haifa 31905, Israel\\(yair@math.haifa.ac.il)}
\date{February 16, 2012. Revised: April 15, 2012}
\maketitle

\begin{abstract}
We consider sequential iterative processes for the common fixed point problem
of families of \textit{cutter operators} on a Hilbert space. These are
operators that have the property that, for any point $x\in\mathcal{H}$, the
hyperplane through $Tx$ whose normal is $x-Tx$ always \textquotedblleft
cuts\textquotedblright\ the space into two half-spaces one of which contains
the point $x$ while the other contains the (assumed nonempty) fixed point set
of $T.$ We define and study generalized relaxations and extrapolation of
cutter operators and construct extrapolated cyclic cutter operators. In this
framework we investigate the Dos Santos local acceleration method in a unified
manner and adopt it to a composition of cutters. For these we conduct
convergence analysis of successive iteration algorithms.



\end{abstract}

\section{Introduction\label{sec:intro}}

Our point of departure that motivates us in this work is a local acceleration
technique of Cimmino's \cite{Cim38} well-known simultaneous projection method
for linear equations. This technique is referred to in the literature as the
\textit{Dos Santos }(DS)\textit{\ }method, see Dos Santos \cite{DSa87} and
Bauschke and Borwein \cite[Section 7]{BB96}, although Dos Santos attributes
it, in the linear case, to De Pierro's Ph.D. Thesis \cite{DPi81}. The method
essentially uses the line through each pair of consecutive Cimmino iterates
and chooses the point on this line which is closest to the solution $x^{\ast}$
of the linear system $Ax=b$. The nice thing about it is that existence of the
solution of the linear system must be assumed, but the method does not need
the solution point $x^{\ast}$ in order to proceed with the locally accelerated
DS iterative process. This approach was also used by Appleby and Smolarski
\cite{AS05}. On the other hand, while trying to be as close as possible to the
solution point $x^{\ast}$ in each iteration, the method is not yet known to
guarantee overall acceleration of the process. Therefore, we call it a
\textit{local acceleration }technique. In all the above references the DS
method works for \textit{simultaneous projection methods} and our first
question was whether it can also work with \textit{sequential projection
methods}? Once we discovered that this is possible, the next natural question
for sequential locally accelerated DS iterative process, is how far can the
principle of the DS method be upgraded from the linear equations model? Can it
work for closed and convex sets feasibility problems? I.e., can the locally
accelerated DS method be preserved if orthogonal projections onto hyperplanes
are replaced by metric projections onto closed and convex sets? Furthermore,
can the latter be replaced by subgradient projectors onto closed and convex
sets in a valid locally accelerated DS method? Finally, can the theory be
extended to handle common fixed point problems? If so, for which classes of operators?

In this study we answer these questions by focusing on the class of operators
$T:\mathcal{H}\rightarrow\mathcal{H}$, where $\mathcal{H}$ is a Hilbert space,
that have the property that, for any $x\in\mathcal{H}$, the hyperplane through
$Tx$ whose normal is $x-Tx$ always \textquotedblleft cuts\textquotedblright%
\ the space into two half-spaces one of which contains the point $x$ while the
other contains the (assumed nonempty) fixed point set of $T.$ This explains
the name \textit{cutter operators} or \textit{cutters} that we introduce here.
These operators, introduced and investigated by Bauschke and Combettes
\cite[Definition 2.2]{BC01} and by Combettes \cite{Com01}, play an important
role in optimization and feasibility theory since many commonly used operators
are actually cutters. We define generalized relaxations and extrapolation of
cutter operators and construct \textit{extrapolated cyclic cutter operators.
}For this cyclic extrapolated cutters we present convergence results of
successive iteration processes for common fixed point problems.

Finally we show that these iterative algorithmic frameworks can handle
sequential locally accelerated DS iterative processes, thus, cover some of the
earlier results about such methods and present some new ones.

The paper is organized as follows. In Section \ref{sec:prelim} we give the
definition of cutter operators and bring some of their properties that will be
used here. Section \ref{sec: main} presents the main convergence results.
Applications to specific convex feasibility problems, which show how the
locally accelerated DS iterative processes follow from our general convergence
results, are furnished in Section \ref{sect:applications}.

\section{Preliminaries\label{sec:prelim}}

Let $\mathcal{H}$ be a real Hilbert space with an inner product $\langle
\cdot,\cdot\rangle$ and with a norm $\Vert\cdot\Vert$. Given $x,y\in
\mathcal{H}$ we denote%
\begin{equation}
H(x,y):=\left\{  u\in\mathcal{H}\mid\left\langle u-y,x-y\right\rangle
\leq0\right\}  .
\end{equation}

\begin{definition}
\label{d-sep}%
\rm\
An operator $T:\mathcal{H}\rightarrow\mathcal{H}$ is called a \textit{cutter
operator} or, in short, a \textit{cutter}%
\begin{equation}
\text{if }\operatorname*{Fix}T\subseteq H(x,Tx)\text{ \ for all }%
x\in\mathcal{H},
\end{equation}
where $\operatorname*{Fix}T$ is the fixed point set of $T$, equivalently,%
\begin{equation}
q\in\operatorname*{Fix}T\text{ implies that }\left\langle
Tx-x,Tx-q\right\rangle \leq0\text{ \ for all }x\in\mathcal{H}.
\label{def directed}%
\end{equation}

\end{definition}

\bigskip

The inequality in (\ref{def directed}) can be written equivalently in the form%
\begin{equation}
\left\langle Tx-x,q-x\right\rangle \geq\Vert Tx-x\Vert^{2}\text{.}
\label{e-sep1}%
\end{equation}
The class of cutter operators is denoted by $\mathcal{T}$, i.e.,%
\begin{equation}
\mathcal{T}:=\left\{  T:\mathcal{H}\rightarrow\mathcal{H}\mid
\operatorname*{Fix}T\subseteq H(x,Tx)\text{ \ for all }x\in\mathcal{H}%
\right\}  .
\end{equation}

The class $\mathcal{T}$ of operators was introduced and investigated by
Bauschke and Combettes in \cite[Definition 2.2]{BC01} and by Combettes in
\cite{Com01}. Yamada and Ogura \cite{YO04} and Mainge \cite{Mai10} named the
cutters \textit{firmly quasi-nonexpansive} operators. These operators were
named \textit{directed operators }in Zaknoon \cite{Zak03} and further employed
under this name by Segal \cite{Seg08} and Censor and Segal \cite{CS08, CS08a,
CS09}. Cegielski \cite[Definition 2.1]{Ceg10} named and studied these
operators as \textit{separating operators}. Since both \textit{directed
}and\textit{\ separating }are key words of other, widely-used, mathematical
entities we decided in \cite{CC11} to use the term \textit{cutter operators}.
This name can be justified by the fact that the bounding hyperplane of
$H(x,Tx)$ \textquotedblleft cuts\textquotedblright\ the space into two
half-spaces, one which contains the point $x$ while the other contains the set
$\operatorname*{Fix}T$. We recall definitions and results on cutter operators
and their properties as they appear in \cite[Proposition 2.4]{BC01} and
\cite{Com01}, which are also sources for further references.

Bauschke and Combettes \cite{BC01} showed the following:

\begin{itemize}
\item[(i)] The set of all fixed points of a cutter operator with nonempty
$\operatorname*{Fix}T$ is a closed and convex subset of $\mathcal{H}$, because
$\operatorname*{Fix}T=\cap_{x\in\mathcal{H}}H(x,Tx)$.
\end{itemize}

Denoting by $I$ the identity operator,%
\begin{equation}
\text{if }T\in\mathcal{T}\text{ then }I+\lambda(T-I)\in\mathcal{T}\text{ \ for
all }\lambda\in\lbrack0,1]. \label{BCresult}%
\end{equation}
This class of operators is fundamental because many common types of operators
arising in convex optimization belong to the class and because it allows a
complete characterization of Fej\'{e}r-monotonicity \cite[Proposition
2.7]{BC01}. The localization of fixed points is discussed by Goebel and Reich
in \cite[pp. 43--44]{GR84}. In particular, it is shown there that a firmly
nonexpansive operator, namely, an operator $T:\mathcal{H}\rightarrow
\mathcal{H}$ that fulfills%
\begin{equation}
\left\Vert Tx-Ty\right\Vert ^{2}\leq\left\langle Tx-Ty,x-y\right\rangle \text{
for all }x,y\in\mathcal{H},
\end{equation}
and has a fixed point, satisfies (\ref{def directed}) and is, therefore, a
cutter operator. The class of cutter operators, includes additionally,
according to \cite[Proposition 2.3]{BC01}, among others, the resolvents of a
maximal monotone operators, the orthogonal projections and the subgradient
projectors. Another family of cutters appeared recently in Censor and Segal
\cite[Definition 2.7]{CS08a}. Note that every cutter operator belongs to the
class of operators $\mathcal{F}^{0},$ defined by Crombez \cite[p. 161]{Cro05},%
\begin{equation}
\mathcal{F}^{0}:=\left\{  T:\mathcal{H}\rightarrow\mathcal{H}\mid\left\Vert
Tx-q\right\Vert \leq\left\Vert x-q\right\Vert \text{ for all }q\in
\operatorname*{Fix}T\text{ and }x\in\mathcal{H}\right\}  \text{,} \label{PC}%
\end{equation}
whose elements are called elsewhere quasi-nonexpansive or paracontracting
operators. An example of a quasi-nonexpansive operator $T:\mathcal{H}%
\rightarrow\mathcal{H}$ is a nonexpansive one, i.e., an operator satisfying
$\Vert Tx-Ty\Vert\leq\Vert x-y\Vert$ for all $x,y\in\mathcal{H}$, with
$\operatorname*{Fix}T\neq\emptyset$.

\begin{definition}%
\rm\
Let $T:\mathcal{H}\rightarrow\mathcal{H}$ and let $\lambda\in(0,2)$. We call
the operator $T_{\lambda}:=I+(1-\lambda)T$ a \textit{relaxation} of
$T$\texttt{.}
\end{definition}

\begin{definition}%
\rm\
We say that an operator $T:\mathcal{H}\rightarrow\mathcal{H}$ with
$\operatorname*{Fix}T\neq\emptyset$ is \textit{strictly quasi-nonexpansive}%
\texttt{\ }if%
\begin{equation}
\Vert Tx-z\Vert<\Vert x-z\Vert
\end{equation}
for all $x\notin\operatorname*{Fix}T$ and for all $z\in\operatorname*{Fix}T$.
We say that $T$ is $\alpha$-\textit{strongly quasi-nonexpansive}%
,\texttt{\ }where $\alpha>0$, or, in short, \textit{strongly
quasi-nonexpansive} if%
\begin{equation}
\Vert Tx-z\Vert^{2}\leq\Vert x-z\Vert^{2}-\alpha\Vert Tx-x\Vert^{2}%
\end{equation}
for all $x\in\mathcal{H}$ and for all $z\in\operatorname*{Fix}T$.
\end{definition}

It is well-known that a relaxation of a cutter operator is strongly
quasi-nonexpansive (see \cite[Proposition 2.3(ii)]{Com01}). Since the converse
implication is also true we have the following result.

\begin{lemma}
\label{l-cut-SQNE}Let $T:\mathcal{H}\rightarrow\mathcal{H}$ be an operator
which has a fixed point and let $\lambda\in(0,2)$. Then $T$ is a cutter if and
only if its relaxation $T_{\lambda}$ is $(2-\lambda)/\lambda$-strongly quasi-nonexpansive.
\end{lemma}

\begin{proof}
Since%
\begin{equation}
T_{\lambda}x-x=\lambda(Tx-x)\text{,}%
\end{equation}
we have, by the properties of the inner product,%
\begin{align}
&  \Vert T_{\lambda}x-z\Vert^{2}-\Vert x-z\Vert^{2}+\frac{2-\lambda}{\lambda
}\Vert T_{\lambda}x-x\Vert^{2}\nonumber\\
&  =\Vert x-z+\lambda(Tx-x)\Vert^{2}-\Vert x-z\Vert^{2}+\lambda(2-\lambda
)\Vert Tx-x\Vert^{2}\nonumber\\
&  =2\lambda(\Vert Tx-x\Vert^{2}-\langle z-x,Tx-x\rangle)=2\lambda\langle
z-Tx,x-Tx\rangle,
\end{align}
for all $x\in\mathcal{H}$ and for all $z\in\operatorname*{Fix}T$, from which
the required result follows.
\end{proof}

\begin{definition}%
\rm\
We say that an operator $T:\mathcal{H}\rightarrow\mathcal{H}$ is
\textit{demiclosed} at $0$ if for any weakly converging sequence
$\{x^{k}\}_{k=0}^{\infty},$ $x^{k}\rightharpoonup y\in\mathcal{H}$ as
$k\rightarrow\infty$, with $Tx^{k}\rightarrow0$ as $k\rightarrow\infty,$ we
have $Ty=0$.
\end{definition}

It is well-known that for a nonexpansive operator $T:\mathcal{H}%
\rightarrow\mathcal{H}$, the operator $T-I$ is demiclosed at $0,$ see Opial
\cite[Lemma 2]{Opi67}.

\begin{definition}%
\rm\
We say that an operator $T:\mathcal{H}\rightarrow\mathcal{H}$ is
\textit{asymptotically regular} if%
\begin{equation}
\Vert T^{k+1}x-T^{k}x\Vert\rightarrow0,\text{ as }k\rightarrow\infty,
\end{equation}
for all $x\in\mathcal{H}$.
\end{definition}

It is well-known that if $T$ is a nonexpansive and asymptotically regular
operator with $\operatorname*{Fix}T\neq\emptyset$ then, for any $x\in
\mathcal{H}$, the sequence $\{T^{k}x\}_{k=0}^{\infty}$ converges weakly to a
fixed point of $T,$ see \cite[Theorem 1]{Opi67}.

\section{Main results\label{sec: main}}

We deal in this paper with a finite family of cutter operators $U_{i}%
:\mathcal{H}\rightarrow\mathcal{H}$, $i=1,2,\ldots,m$, with $\bigcap_{i=1}%
^{m}\operatorname*{Fix}U_{i}\neq\emptyset$ and with compositions of $U_{i}$,
$i=1,2,\ldots,m$. We propose local acceleration techniques for algorithms
which apply this operation. For an operator $U:\mathcal{H}\rightarrow
\mathcal{H}$ we define the operator $U_{\sigma,\lambda}:\mathcal{H}%
\rightarrow\mathcal{H}$ by%
\begin{equation}
U_{\sigma,\lambda}x:=x+\lambda\sigma(x)(Ux-x)\text{, } \label{e-Tx}%
\end{equation}
where $\lambda\in(0,2)$ is a \textit{relaxation parameter} and $\sigma
:\mathcal{H}\rightarrow(0,+\infty)$ is a \textit{step size function}. We call
the operator $U_{\sigma,\lambda}$ the \textit{generalized relaxation} of $U$
(cf. \cite[Section 1]{Ceg10} and \cite[Definition 9.16]{CC11}). A generalized
relaxation $U_{\sigma,\lambda}$ of $U$ with $\lambda\sigma(x)\geq1$ for all
$x\in\mathcal{H}$ is called an \textit{extrapolation} of $U$. Of course, for
$x\in\operatorname*{Fix}U$, we can set $\sigma(x)$ arbitrarily, e.g.,
$\sigma(x)=1$. Denoting $U_{\sigma}:=U_{\sigma,1}$, it is clear that%
\begin{equation}
U_{\sigma,\lambda}x=x+\lambda(U_{\sigma}x-x), \label{e-Tx2}%
\end{equation}
and that $\operatorname*{Fix}U_{\sigma,\lambda}=\operatorname*{Fix}U_{\sigma
}=\operatorname*{Fix}U$ (note that $\sigma(x)>0$ for all $x\in\mathcal{H}$).

Let $U_{i}:\mathcal{H}\rightarrow\mathcal{H}$ be a cutter operator,
$i=1,2,\ldots,m$, with $\bigcap_{i=1}^{m}\operatorname*{Fix}U_{i}\neq
\emptyset$. Define the operator $U:\mathcal{H}\rightarrow\mathcal{H}$ as the
composition%
\begin{equation}
U:=U_{m}U_{m-1}\cdots U_{1}. \label{eq:cyclic-op}%
\end{equation}
Since any cutter operator is strongly quasi-nonexpansive, (see Lemma
\ref{l-cut-SQNE}), $\operatorname*{Fix}U=\bigcap_{i=1}^{m}\operatorname*{Fix}%
U_{i}$ and $U$ is strongly quasi-nonexpansive, see \cite[Proposition
2.10]{BB96}. Consequently, $U$ is asymptotically regular, see \cite[Corollary
2.8]{BB96}, and, if $U$ is nonexpansive then any sequence $\{x^{k}%
\}_{k=0}^{\infty},$ generated by the recurrence (Piccard iteration)%
\begin{equation}
x^{0}\in\mathcal{H}\text{ is arbitrary, and }x^{k+1}=Ux^{k}\text{, for all
}k\geq0,\text{ } \label{e-cycl-sep}%
\end{equation}
converges weakly to a fixed point of $U$, see \cite[Theorem 1]{Opi67}. We
extend this convergence result to the generalized relaxation of $U$, defined
by (\ref{e-Tx}). We call an operator $U$ of the form (\ref{eq:cyclic-op}) a
\textit{cyclic cutter}. Contrary to the simultaneous cutter, a cyclic cutter
needs not to be a cutter. This \textquotedblleft contradiction of
terms\textquotedblright\ resembles the situation with the, so-called,
\textquotedblleft subgradient projection\textquotedblright\ onto a convex set
which needs not to be a projection onto the set because it needs not be an
element of the set.

In order to prove our convergence result for the generalized relaxation of $U
$, defined by (\ref{e-Tx}), let $S_{0}:=I$ and $S_{i}:=U_{i}U_{i-1}\cdots
U_{1}$, $i=1,2,\ldots,m$. Of course, $U=S_{m}$. Further, denote%
\begin{equation}
u^{0}=x,\ u^{i}=U_{i}u^{i-1}\text{ \ and \ }y^{i}=u^{i}-u^{i-1},\text{ \ for
all \ }i=1,2,\ldots,m. \label{e-yi}%
\end{equation}
According to this notation, used throughout this paper, we have $u^{i}=S_{i}x
$ and $\sum_{j=i}^{m}y^{j}=Ux-S_{i-1}x$, $i=1,2,\ldots,m$, in particular,
$\sum_{j=1}^{m}y^{j}=Ux-x$.

\begin{lemma}
\label{l-Sx-x}If $U_{i}:\mathcal{H}\rightarrow\mathcal{H}$, $i=1,2,\ldots,m$,
are cutter operators such that $\bigcap_{i=1}^{m}\operatorname*{Fix}U_{i}%
\neq\emptyset$, then for any $z\in\bigcap_{i=1}^{m}\operatorname*{Fix}U_{i}$
the following inequalities hold%
\begin{equation}
\langle Ux-x,z-x\rangle\geq\sum_{i=1}^{m}\langle y^{i}+y^{i+1}+\cdots
+y^{m},y^{i}\rangle\geq\frac{1}{2}\sum_{i=1}^{m}\Vert y^{i}\Vert^{2}\geq
\frac{1}{2m}\left\Vert \sum_{i=1}^{m}y^{i}\right\Vert ^{2}\text{,}
\label{e-main2}%
\end{equation}
for all $x\in\mathcal{H}$.
\end{lemma}

\begin{proof}
We prove the first inequality in (\ref{e-main2}) by induction on $m$. For
$m=1$ it follows directly from Definition \ref{d-sep}. Suppose that the first
inequality in (\ref{e-main2}) is true for some $m=t$. Define $V_{1}:=I$,
$V_{i}:=U_{i}U_{i-1}\cdots U_{2}$ for $i=2,3,\ldots,t+1$, $v^{1}:=h$, where
$h$ is an arbitrary element of $\mathcal{H}$,$\ v^{i}:=U_{i}v^{i-1}$ and
$z^{i}:=v^{i}-v^{i-1}$, $i=2,3,\ldots,t+1$. If we set $h=U_{1}x$ then, of
course, $S_{i}x=V_{i}h$, $u^{i}=v^{i}$ and $y^{i}=z^{i}$, $i=2,3,\ldots,t+1$.
It follows from the induction hypothesis that%
\begin{equation}
\langle V_{t+1}h-h,z-h\rangle\geq\sum_{i=2}^{t+1}\langle z^{i}+z^{i+1}%
+\cdots+z^{t+1},z^{i}\rangle
\end{equation}
for all $h\in\mathcal{H}$ and $z\in\cap_{i=2}^{t+1}\operatorname*{Fix}U_{i}$.
Thus, if $h=U_{1}x$ then, for all $x\in\mathcal{H}$ and $z\in\cap_{i=1}%
^{t+1}\operatorname*{Fix}U_{i}$, we obtain%
\begin{align}
&  \langle S_{t+1}x-x,z-x\rangle=\langle V_{t+1}h-x,z-x\rangle\nonumber\\
&  =\langle V_{t+1}h-h,z-x\rangle+\langle U_{1}x-x,z-x\rangle\nonumber\\
&  \geq\langle V_{t+1}h-h,z-h\rangle+\langle V_{t+1}h-h,h-x\rangle+\Vert
y^{1}\Vert^{2}\nonumber\\
&  \geq\sum_{i=2}^{t+1}\langle z^{i}+z^{i+1}+\cdots+z^{t+1},z^{i}%
\rangle+\left\langle \sum_{i=2}^{t+1}z^{i},h-x\right\rangle +\Vert y^{1}%
\Vert^{2}\nonumber\\
&  =\sum_{i=2}^{t+1}\langle y^{i}+y^{i+1}+\cdots+y^{t+1},y^{i}\rangle+\langle
y^{2}+y^{3}+\cdots+y^{t+1},y^{1}\rangle+\Vert y^{1}\Vert^{2}\nonumber\\
&  =\sum_{i=1}^{t+1}\langle y^{i}+y^{i+1}+\cdots+y^{t+1},y^{i}\rangle\text{.}%
\end{align}
Therefore, the first inequality in (\ref{e-main2}) is true for $m=t+1,$ and
the induction is complete. The second inequality follows from%
\begin{equation}
\sum_{i=1}^{m}\langle y^{i}+y^{i+1}+\cdots+y^{m},y^{i}\rangle-\frac{1}{2}%
\sum_{i=1}^{m}\Vert y^{i}\Vert^{2}=\frac{1}{2}\Vert\sum_{i=1}^{m}y^{i}%
\Vert^{2}\text{.} \label{e-sum-yi}%
\end{equation}
The third inequality in (\ref{e-main2}) follows from the convexity of the
function $\Vert\cdot\Vert^{2}$.
\end{proof}

\bigskip

\medskip Define the step size function $\sigma_{\max}:\mathcal{H}%
\rightarrow(0,+\infty)$ by%
\begin{equation}
\sigma_{\max}(x):=\left\{
\begin{array}
[c]{ll}%
\displaystyle\frac{\sum_{i=1}^{m}\langle Ux-S_{i-1}x,S_{i}x-S_{i-1}x\rangle
}{\Vert Ux-x\Vert^{2}}, & \text{for \ }x\notin\operatorname*{Fix}U,\\
1, & \text{for \ }x\in\operatorname*{Fix}U.
\end{array}
\right.  \label{e-sigma1}%
\end{equation}
If we set $y^{i}=S_{i}x-S_{i-1}x$, $i=1,2,\ldots,m$, (compare with
(\ref{e-yi})) in Lemma \ref{l-Sx-x}, then we obtain for $x\notin
\operatorname*{Fix}U$%
\begin{equation}
\sigma_{\max}(x)\geq\frac{\frac{1}{2}\sum_{i=1}^{m}\Vert S_{i}x-S_{i-1}%
x\Vert^{2}}{\Vert Ux-x\Vert^{2}}\geq\frac{1}{2m}\text{.} \label{e-sigma1a}%
\end{equation}

\begin{lemma}
\label{l-T-sep}Let $U_{i}:\mathcal{H}\rightarrow\mathcal{H}$ be cutter
operators, $i=1,2,\ldots,m$, with $\bigcap_{i=1}^{m}\operatorname*{Fix}%
U_{i}\neq\emptyset$. The operator $U_{\sigma}:=U_{\sigma,1},$ defined by
$\mathrm{(\ref{e-Tx})}$, where the step size function $\sigma:=\sigma_{\max}$
is given by $\mathrm{(\ref{e-sigma1})}$, is a cutter.
\end{lemma}

\begin{proof}
Taking $z\in\bigcap_{i=1}^{m}\operatorname*{Fix}U_{i},$ the first inequality
in (\ref{e-main2}) with $y^{i}=S_{i}x-S_{i-1}x$, $i=1,2,\ldots,m$, can be
rewritten, for all $x\in\mathcal{H}$, in the form%
\begin{equation}
\Vert Ux-x\Vert^{2}\sigma(x)\leq\langle Ux-x,z-x\rangle\text{.} \label{e-Ux}%
\end{equation}
Therefore, for the operator $U_{\sigma}x,$ defined by (\ref{e-Tx}) with the
step size $\sigma(x)$ as in (\ref{e-sigma1}), we have%
\begin{equation}
\langle U_{\sigma}x-x,z-x\rangle=\sigma(x)\langle Ux-x,z-x\rangle\geq
\sigma^{2}(x)\Vert Ux-x\Vert^{2}=\Vert U_{\sigma}x-x\Vert^{2},
\label{eq:snake}%
\end{equation}
and the proof is complete.
\end{proof}

\begin{theorem}
\label{t-main}Let $U_{i}:\mathcal{H}\rightarrow\mathcal{H}$, $i=1,2,\ldots,m$,
be a cutter operator with $\bigcap_{i=1}^{m}\operatorname*{Fix}U_{i}%
\neq\emptyset$. Let the sequence $\{x^{k}\}_{k=0}^{\infty}$ be defined by
\begin{equation}
x^{k+1}=U_{\sigma,\lambda_{k}}x^{k}\text{,} \label{e-rec}%
\end{equation}
where $U_{\sigma,\lambda}$ is given by $\mathrm{(\ref{e-Tx})}$, $\lambda
_{k}\in\lbrack\varepsilon,2-\varepsilon]$ for an arbitrary constant
$\varepsilon\in(0,1)$, and $\sigma:=\sigma_{\max}$ is given by
$\mathrm{(\ref{e-sigma1})}$. Then%
\begin{align}
\Vert x^{k+1}-z\Vert^{2}  &  \leq\Vert x^{k}-z\Vert^{2}-\frac{\lambda
_{k}(2-\lambda_{k})}{4}\frac{\sum_{i=1}^{m}\Vert S_{i}x^{k}-S_{i-1}x^{k}%
\Vert^{2}}{\Vert Ux^{k}-x^{k}\Vert^{2}}\nonumber\\
&  \leq\Vert x^{k}-z\Vert^{2}-\frac{\lambda_{k}(2-\lambda_{k})}{4m^{2}}\Vert
Ux^{k}-x^{k}\Vert^{2}\text{,} \label{e-xk}%
\end{align}
for all $x^{k}\notin\operatorname*{Fix}U$ and for all $z\in\operatorname*{Fix}%
U$. Consequently, $x^{k}\rightharpoonup x^{\ast}\in\operatorname*{Fix}U$ as
$k\rightarrow\infty$, if one of the following conditions is satisfied:

\begin{itemize}
\item[$\mathrm{(i)}$] $U-I$ is demiclosed at $0$, or

\item[$\mathrm{(ii)}$] $U_{i}-I$ are demiclosed at $0$, for all $i=1,2,\ldots
,m$.
\end{itemize}
\end{theorem}

\begin{proof}
Let $z\in\operatorname*{Fix}U$. By inequality (\ref{e-sigma1a}), we have%
\begin{align}
\Vert x^{k+1}-z\Vert^{2}  &  =\Vert U_{\sigma,\lambda_{k}}x^{k}-z\Vert
^{2}=\Vert x^{k}-z+\lambda_{k}(U_{\sigma}x^{k}-x^{k})\Vert^{2}\nonumber\\
&  =\Vert x^{k}-z\Vert^{2}+\lambda_{k}^{2}\Vert U_{\sigma}x^{k}-x^{k}\Vert
^{2}-2\lambda_{k}\langle U_{\sigma}x^{k}-x^{k},z-x^{k}\rangle\nonumber\\
&  \leq\Vert x^{k}-z\Vert^{2}+\lambda_{k}^{2}\Vert U_{\sigma}x^{k}-x^{k}%
\Vert^{2}-2\lambda_{k}\Vert U_{\sigma}x^{k}-x^{k}\Vert^{2}\nonumber\\
&  =\Vert x^{k}-z\Vert^{2}-\lambda_{k}(2-\lambda_{k})\Vert U_{\sigma}%
x^{k}-x^{k}\Vert^{2}\nonumber\\
&  =\Vert x^{k}-z\Vert^{2}-\lambda_{k}(2-\lambda_{k})\sigma^{2}(x^{k})\Vert
Ux^{k}-x^{k}\Vert^{2}\nonumber\\
&  \leq\Vert x^{k}-z\Vert^{2}-\frac{\lambda_{k}(2-\lambda_{k})}{4}\frac
{(\sum_{i=1}^{m}\Vert S_{i}x^{k}-S_{i-1}x^{k}\Vert^{2})^{2}}{\Vert
Ux^{k}-x^{k}\Vert^{2}}\nonumber\\
&  \leq\Vert x^{k}-z\Vert^{2}-\frac{\lambda_{k}(2-\lambda_{k})}{4m^{2}}\Vert
Ux^{k}-x^{k}\Vert^{2}\text{,}%
\end{align}
for all $x^{k}\notin\operatorname*{Fix}U$. Therefore, $\{\Vert x^{k}%
-z\Vert\}_{k=0}^{\infty}$ is decreasing, $\{x^{k}\}_{k=0}^{\infty}$ is
bounded, and as $k\rightarrow\infty,$ we have%
\begin{equation}
\Vert Ux^{k}-x^{k}\Vert\rightarrow0\text{,} \label{e-Uxk}%
\end{equation}
and%
\begin{equation}
\Vert S_{i}x^{k}-S_{i-1}x^{k}\Vert\rightarrow0\text{,} \label{e-Sixk}%
\end{equation}
for all $i=1,2,\ldots,m$. Let $x^{\ast}\in\mathcal{H}$ be a weak cluster point
of $\{x^{k}\}_{k=0}^{\infty}$ and $\{x^{n_{k}}\}_{k=0}^{\infty}\subset
\{x^{k}\}_{k=0}^{\infty}$ be a subsequence which converges weakly to $x^{\ast
}$.

(i) Suppose that $U-I$ is demiclosed at $0$. Condition (\ref{e-Uxk}) yields
that $x^{\ast}\in\operatorname*{Fix}U$. The convergence of the whole sequence
$\{x^{k}\}_{k=0}^{\infty}$ to $x^{\ast}$ follows now from \cite[Theorem 2.16
(ii)]{BB96}.

(ii) Suppose that $U_{i}-I$ are demiclosed at $0$, for all $i=1,2,\ldots,m$.
Condition (\ref{e-Sixk}) for $i=1$ yields%
\begin{equation}
\Vert U_{1}x^{n_{k}}-x^{n_{k}}\Vert=\Vert S_{1}x^{n_{k}}-S_{0}x^{n_{k}}%
\Vert\rightarrow0\text{ as }k\rightarrow\infty\text{.}%
\end{equation}
Due to demiclosedness of $U_{1}-I$ at $0$, we have that $U_{1}x^{\ast}%
=x^{\ast}$. Since%
\begin{equation}
\Vert(U_{1}x^{n_{k}}-U_{1}x^{\ast})-(x^{n_{k}}-x^{\ast})\Vert=\Vert
U_{1}x^{n_{k}}-x^{n_{k}}\Vert\rightarrow0\text{ as }k\rightarrow\infty\text{,}%
\end{equation}
and $x^{n_{k}}\rightharpoonup x^{\ast}$, as $k\rightarrow\infty$, we have that
$U_{1}x^{n_{k}}\rightharpoonup U_{1}x^{\ast}=x^{\ast}$ as $k\rightarrow\infty
$. Since $U_{2}-I$ is demiclosed at $0$, condition (\ref{e-Sixk}) for $i=2$
implies that $U_{2}x^{\ast}=x^{\ast}$. In a similar way we obtain that
$U_{i}x^{\ast}=x^{\ast}$ for $i=3,4,\ldots,m$. Therefore,%
\begin{equation}
Ux^{\ast}=S_{m}x^{\ast}=U_{m}U_{m-1}\cdots U_{1}x^{\ast}=x^{\ast}\text{.}%
\end{equation}
We conclude that the subsequence $\{x^{n_{k}}\}_{k=0}^{\infty}$ converges
weakly to a fixed point of the operator $U$. The weak convergence of the whole
sequence $\{x^{k}\}_{k=0}^{\infty}$ to $x^{\ast}\in\operatorname*{Fix}U$
follows now from \cite[Theorem 2.16 (ii)]{BB96}.
\end{proof}

\begin{remark}%
\rm\
\label{r-sum-yi}Note that%
\begin{equation}
\sum_{i=1}^{m}\langle y^{i}+y^{i+1}+\cdots+y^{m},y^{i}\rangle=\sum_{i=1}%
^{m}\langle y^{1}+y^{2}+\cdots+y^{i},y^{i}\rangle\text{.}%
\end{equation}
Therefore, the step size $\sigma_{\max}(x)$ for $x\notin\operatorname*{Fix}U$,
given by (\ref{e-sigma1}), can be equivalently written as%
\begin{equation}
\sigma_{\max}(x)=\frac{\sum_{i=1}^{m}\langle S_{i}x-x,S_{i}x-S_{i-1}x\rangle
}{\Vert Ux-x\Vert^{2}}\text{.} \label{e-sigma2}%
\end{equation}
Furthermore, by (\ref{e-sum-yi}), we have%
\begin{equation}
\sigma_{\max}(x)=\frac{\Vert Ux-x\Vert^{2}+\sum_{i=1}^{m}\left\Vert
S_{i}x-S_{i-1}x\right\Vert ^{2}}{2\Vert Ux-x\Vert^{2}}\text{.}
\label{e-sigma2a}%
\end{equation}

\end{remark}

\begin{remark}%
\rm\
\label{r-s1/2}Theorem $\mathrm{\ref{t-main}}$ remains true if we suppose that
the step size function $\sigma$ is an arbitrary function satisfying the
inequalities%
\begin{equation}
\alpha\frac{\sum_{i=1}^{m}\Vert S_{i}x-S_{i-1}x\Vert^{2}}{\Vert Ux-x\Vert^{2}%
}\leq\sigma(x)\leq\sigma_{\max}(x)\text{,}%
\end{equation}

where $\alpha\in(0,1/2]$ and $x\notin\operatorname*{Fix}U$. The existence of
such a function follows from Lemma $\mathrm{\ref{l-Sx-x}}$.
\end{remark}

\begin{remark}
\label{r-s-max}%
\rm\
Even if we take $\lambda=2$, the generalized relaxation $U_{\sigma_{\max
},\lambda}$ needs not to be an extrapolation of $U$ because the inequality
$\sigma_{\max}(x)\geq1/2$ is not guaranteed. We only know that $\sigma_{\max
}(x)\geq1/(2m)$ (see (\ref{e-sigma1a})). It is known, however, that $U$ is
$1/m$-SQNE (see \cite[Proposition 1(d)(iii)]{YO04}), consequently, it is a
$2m/(m+1)$-relaxed cutter (see Lemma \ref{l-cut-SQNE}). Therefore,
$U_{\frac{1+m}{2m}}$ is a cutter. Since $U_{\sigma}$ is a cutter (see Lemma
\ref{l-T-sep}), one can easily check that $U_{\sigma}$ is also a cutter,
where
\begin{equation}
\sigma=\max((1+m)/(2m),\sigma_{\max})\text{.} \label{e-sigma}%
\end{equation}
Therefore, $U_{\sigma,\lambda}$ with $\lambda\in(2m/(m+1),2]$ is an
extrapolation of $U$ and one can expect that an application of $U_{\sigma
,\lambda}$ leads in practice to a local acceleration of the convergence of
sequences generated by the recurrence $x^{k+1}=U_{\sigma,\lambda_{k}}x^{k},$
in comparison to the classical cyclic cutter method $x^{k+1}=Ux^{k}$. Note,
that if we apply the step size $\sigma$ given by (\ref{e-sigma}), then Theorem
\ref{t-main} remains true, because $\left\Vert U_{\sigma_{\max}}x-x\right\Vert
\leq\left\Vert U_{\sigma}x-x\right\Vert $.
\end{remark}

\section{Applications\label{sect:applications}}

In this section we show how our general results unify, generalize and extend
several existing local acceleration schemes. We consider two kinds of
operators $U_{i}$ satisfying the assumptions of Section \ref{sec: main}: (i)
$U_{i}$ is the metric projection onto a closed and convex set $C_{i}%
\subset\mathcal{H}$, and (ii) $U_{i}$ is a subgradient projection onto a set
$C_{i}=\{x\in\mathcal{H}\mid c_{i}(x)\leq0\}$, where $c_{i}:\mathcal{H}%
\rightarrow%
\mathbb{R}
$ is a continuous and convex function, $i=1,2,\ldots,m$. Recall that
$\operatorname*{Fix}P_{C_{i}}=C_{i}$ and that $P_{C_{i}}$ is a firmly
nonexpansive operator (see, e.g., Zarantonello \cite[Lemma 1.2]{Zar71}),
consequently, $P_{C_{i}}$ is both nonexpansive (by the Cauchy--Schwarz
inequality) and a cutter operator (see \cite[Lemma 2.4 (ii)]{BB96}).
Furthermore, Opial's demiclosedness principle yields that the operator
$P_{C_{i}}-I$ is demiclosed at $0$ (see \cite[Lemma 2]{Opi67}).

First we consider the case $U_{i}=P_{C_{i}}$. If $C_{i}\subseteq\mathcal{H}$
is a general closed and convex subset, one can apply Theorem \ref{t-main}
directly. If $C_{i}$ is a hyperplane, i.e., $C_{i}=H(a^{i},b_{i}%
):=\{x\in\mathcal{H}\mid\langle a^{i},x\rangle=b_{i}\}$, where $a^{i}%
\in\mathcal{H}$, $a^{i}\neq0$ and $b_{i}\in\mathbb{R}$, then%
\begin{equation}
U_{i}:=P_{C_{i}}x=x-\frac{\langle a^{i},x\rangle-b_{i}}{\Vert a^{i}\Vert^{2}%
}a^{i}\text{.} \label{e-Uix1}%
\end{equation}
In this case an application of Theorem \ref{t-main} will be presented in
Subsection \ref{subsect:seq-kaczmarz}. If $C_{i}$ is a half-space,
$C_{i}=H_{-}(a^{i},b_{i}):=\{x\in\mathcal{H}\mid\langle a^{i},x\rangle\leq
b_{i}\}$, then%
\begin{equation}
U_{i}:=P_{C_{i}}x=x-\frac{(\langle a^{i},x\rangle-b_{i})_{+}}{\Vert a^{i}%
\Vert^{2}}a^{i}\text{,} \label{e-Uix2}%
\end{equation}

where $\alpha_{+}:=\max(0,\alpha)$ for any real number $\alpha$. In this case
an application of Theorem \ref{t-main} will be presented in Subsection
\ref{subsect:seq-inequlities}.

Now consider the case, where $U_{i}$ is a subgradient projection onto
$C_{i}:=\{x\in\mathcal{H}\mid c_{i}(x)\leq0\}$, where $c_{i}:\mathcal{H}%
\rightarrow%
\mathbb{R}
$ is a continuous and convex function. Then we have
\begin{equation}
U_{i}x=P_{c_{i}}x:=\left\{
\begin{array}
[c]{ll}%
\displaystyle x-\frac{(c_{i}(x))_{+}}{\Vert g_{i}(x)\Vert^{2}}g_{i}(x)\text{,}
& \text{if }g_{i}(x)\neq0\text{,}\\
x\text{,} & \text{if }g_{i}(x)=0\text{,}%
\end{array}
\right.  \label{e-Uix3}%
\end{equation}
where $g_{i}(x)\in\partial c_{i}(x):=\{g\in\mathcal{H}\mid\langle
g,y-x\rangle\geq c_{i}(y)-c_{i}(x),$ for all $y\in\mathcal{H}\}$ is a
subgradient of the function $c_{i}$ at the point $x$. It follows from the
definition of the subgradient that $U_{i}$ is a cutter, $i=1,2,\ldots,m$. Note
that $\operatorname*{Fix}U_{i}=C_{i}$, consequently,%
\begin{equation}
\bigcap_{i=1}^{m}\operatorname*{Fix}U_{i}=C\neq\emptyset\text{.}%
\end{equation}

Suppose that the subgradients $g_{i}$ are bounded on bounded subsets, $i\in I$
(this holds if, e.g., $\mathcal{H}=%
\mathbb{R}
^{n}$, see \cite[Corollary 7.9]{BB96}). Then the operator $U_{i}%
-\operatorname*{Id}$ is demiclosed at $0$, $i\in I$. Indeed, let
$x^{k}\rightharpoonup x^{\ast}$ and $\lim_{k\rightarrow\infty}\Vert U_{i}%
x^{k}-x^{k}\Vert=0$. Then we have%
\begin{equation}
\lim_{k\rightarrow\infty}\Vert U_{i}x^{k}-x^{k}\Vert=\lim_{k\rightarrow\infty
}\frac{(c_{i}(x^{k}))_{+}}{\Vert g_{i}(x^{k})\Vert}=0\text{.} \label{e-Uixk}%
\end{equation}
The sequence $\{x^{k}\}_{k=0}^{\infty}$ is bounded due to its weak
convergence. Since a continuous and convex function is locally
Lipschitz-continuous, the subgradients sequence $\{g_{i}(x^{k})\}_{k=0}%
^{\infty}$ is also bounded. Equality (\ref{e-Uixk}) implies now the
convergence $\lim_{k\rightarrow\infty}c_{i}(x^{k})_{+}=0$. Since $c_{i}$ is
weakly lower semi-continuous, we have $c_{i}(x^{\ast})=0$, i.e.,
$U_{i}-\operatorname*{Id}$ is demiclosed at $0$.

The local acceleration schemes of the cyclic projection and the cyclic
subgradient projection methods, referred to in Subsections
\ref{subsect:seq-kaczmarz}, \ref{subsect:seq-inequlities} and
\ref{subsect:seq-csp} below, follow the same basic principle of the Dos Santos
(DS) local acceleration principle referred to in Section \ref{sec:intro}.
Namely, consider the line through two consecutive iterates of the cyclic
method applied to a linear system of equations and find on it a point closest
to the nonempty intersection of the feasibility problem sets.

\subsection{Local acceleration of the sequential Kaczmarz \newline method for
linear equations\label{subsect:seq-kaczmarz}}

Consider a consistent system of linear equations%
\begin{equation}
\langle a^{i},x\rangle=b_{i},\text{ \ }i=1,2,\ldots,m\text{,} \label{e-ls}%
\end{equation}
where $a^{i}\in\mathcal{H}$, $a^{i}\neq0$ and $b_{i}\in\mathbb{R}$, for all
$i=1,2,\ldots,m$. Let $C_{i}=H(a^{i},b_{i}):=\{x\in\mathcal{H}\mid\langle
a^{i},x\rangle=b_{i}\}$, $U_{i}=P_{C_{i}}$ be defined by (\ref{e-Uix1}),
$i=1,2,\ldots,m$, and assume that $C:=\bigcap_{i=1}^{m}C_{i}\neq\emptyset$.
Denote $U:=U_{m}U_{m-1}\cdots U_{1}$. The operator $U$ is nonexpansive as a
composition of nonexpansive operators. The \textit{Kaczmarz method} for
solving a system of equations (\ref{e-ls}) has the form%
\begin{equation}
x^{k+1}=Ux^{k}\text{,} \label{e-KM}%
\end{equation}
where $x^{0}\in\mathcal{H}$ is arbitrary. The method was introduced by
Kaczmarz in 1937 for a square nonsingular system of linear equations in
$\mathbb{R}^{n}$ (see \cite{Kac37}). It is well-known that for any starting
point $x^{0}\in\mathcal{H}$ any sequence generated by the recurrence
(\ref{e-KM}) converges strongly to a fixed point of the operator $U$ (see
\cite[Theorem 1]{GPR67}\textbf{)}. The local acceleration scheme for this
method (\ref{e-KM}) which we propose here is a special case of the iterative
procedure $x^{k+1}=U_{\sigma,\lambda_{k}}x^{k}$, where the operator
$U_{\sigma,\lambda}$ is defined by (\ref{e-Tx}), the step size function
$\sigma:\mathcal{H}\rightarrow(0,+\infty)$ is defined by (\ref{e-sigma2}) and
the relaxation parameter is $\lambda\in(0,2)$. Since%
\begin{align}
\sum_{i=1}^{m}\langle S_{i}x-x,S_{i}x-S_{i-1}x\rangle &  =\sum_{i=1}%
^{m}\langle S_{i}x-x,U_{i}u^{i-1}-u^{i-1}\rangle\nonumber\\
&  =\sum_{i=1}^{m}\langle S_{i}x-x,\frac{b_{i}-\langle a^{i},u^{i-1}\rangle
}{\Vert a^{i}\Vert^{2}}a^{i}\rangle\nonumber\\
&  =\sum_{i=1}^{m}\frac{b_{i}-\langle a^{i},u^{i-1}\rangle}{\Vert a^{i}%
\Vert^{2}}(\langle S_{i}x,a^{i}\rangle-\langle x,a^{i}\rangle)\nonumber\\
&  =\sum_{i=1}^{m}\frac{b_{i}-\langle a^{i},u^{i-1}\rangle}{\Vert a^{i}%
\Vert^{2}}(b_{i}-\langle a^{i},x\rangle)\text{,}%
\end{align}
we obtain the following form for the step size%
\begin{equation}
\sigma_{\max}(x)=\frac{\displaystyle\sum_{i=1}^{m}(b_{i}-\langle
a^{i},x\rangle)\frac{\displaystyle b_{i}-\langle a^{i},u^{i-1}\rangle
}{\displaystyle\Vert a^{i}\Vert^{2}}}{\Vert Ux-x\Vert^{2}}\text{,}
\label{e-sigma4}%
\end{equation}
where $x\notin\operatorname*{Fix}U$. This step size is equivalent to those in
(\ref{e-sigma1}), (\ref{e-sigma2}) and (\ref{e-sigma2a}).

\begin{corollary}
\label{c-accKac}Let $U:=U_{m}U_{m-1}\cdots U_{1}$, where $U_{i}$ is given by
$\mathrm{(\ref{e-Uix1})}$, $i=1,2,\ldots,m$, the sequence $\{x^{k}%
\}_{k=0}^{\infty}\subset\mathcal{H}$ be defined by the recurrence%
\begin{equation}
x^{k+1}=U_{\sigma,\lambda_{k}}x^{k}=x^{k}+\lambda_{k}\sigma(x^{k}%
)(Ux^{k}-x^{k})
\end{equation}
for all $x^{k}\notin\operatorname*{Fix}U$, where $x^{0}\in\mathcal{H}$ is
arbitrary, $\sigma:=\sigma_{\max}$ is defined by $\mathrm{(\ref{e-sigma4})}$
and $\lambda_{k}\in\lbrack\varepsilon,2-\varepsilon]$ for an arbitrary
constant $\varepsilon\in(0,1)$. Then $\{x^{k}\}_{k=0}^{\infty}$ converges
weakly to a solution of the system $\mathrm{(\ref{e-ls})}$.
\end{corollary}

One can prove even the strong convegence of $\{x^{k}\}_{k=0}^{\infty}$ in
Corollary \ref{c-accKac}. The proof will be presented elsewhere. As mentioned
in Remark \ref{r-s-max}, Corollary \ref{c-accKac} remains true, is we set
$\sigma:=\max((m+1)/(2m),\sigma_{\max})$. In this case, the method is an
extrapolation of the Kaczmarz method if $\lambda_{k}\in(2m/(m+1),2]$.

\begin{remark}%
\rm\
\label{r-opt-ss}When $C_{i}$ are hyperplanes and $U_{i}=P_{C_{i}}$,
$i=1,2,\ldots,m$, then, for any $u\in\mathcal{H}$ and for any $z\in C_{i}$, we
have%
\begin{equation}
\langle U_{i}u-u,z-u\rangle=\Vert U_{i}u-u\Vert^{2}\text{,}%
\end{equation}
$i=1,2,\ldots,m$. Therefore, it follows from the proof of Lemma \ref{l-Sx-x}
that the first inequality in (\ref{e-main2}) is, actually, an equality.
Consequently, we have equality in (\ref{e-Ux}) and the operator $U_{\sigma}$
defined by (\ref{e-Tx}) with the step size $\sigma:=\sigma_{\max}$ given by
(\ref{e-sigma4}) has the property%
\begin{equation}
\langle U_{\sigma}x-x,z-U_{\sigma}x\rangle=0\text{,}%
\end{equation}
for all $z\in C$, or, equivalently,%
\begin{equation}
\langle U_{\sigma}x-x,z-x\rangle=\Vert U_{\sigma}x-x\Vert^{2}\text{.}%
\end{equation}
This yields the following nice property of the operator $U_{\sigma}$%
\begin{equation}
\Vert U_{\sigma}x-z\Vert=\min\{\Vert x+\alpha(Ux-x)-z\Vert\mid\alpha
\in\mathbb{R}\}\text{,} \label{e-ts-min}%
\end{equation}
for all $z\in C$. We can expect that this property leads in practice to a
local acceleration of the convergence to a solution of the system
(\ref{e-ls}), of sequences generated by the recurrence $x^{k+1}=U_{\sigma
,\lambda_{k}}x^{k},$ where the operator $U_{\sigma,\lambda}$ is defined by
(\ref{e-Tx}) with the step size given by (\ref{e-sigma4}) and $\lambda_{k}%
\in\lbrack\varepsilon,2-\varepsilon]$ for an arbitrary constant $\varepsilon
\in(0,1)$.
\end{remark}

\subsection{Local acceleration of the sequential cyclic projection method for
linear inequalities\label{subsect:seq-inequlities}}

Given a consistent system of linear inequalities%
\begin{equation}
\langle a^{i},x\rangle\leq b_{i},\text{ \ }i=1,2,\ldots,m\text{,} \label{e-li}%
\end{equation}
where $a^{i}\in\mathcal{H}$, $a^{i}\neq0$ and $b_{i}\in\mathbb{R}$. Let
$C_{i}=H_{-}(a^{i},b_{i}):=\{x\in\mathcal{H}\mid\langle a^{i},x\rangle\leq
b_{i}\}$ and $U_{i}=P_{C_{i}}$, $i=1,2,\ldots,m$, be defined by (\ref{e-Uix2}%
). Assume that $C=\bigcap_{i=1}^{m}C_{i}\neq\emptyset$. Denoting $S_{0}:=I$,
$S_{i}:=U_{i}U_{i-1}\cdots U_{1}$, $i=1,2,\ldots,m$, we have $U=S_{m}$. By the
nonexpansiveness of $U$, Theorem \ref{t-main}(i) guarantees the weak
convergence of sequences generated by the recurrence $x^{k+1}=U_{\sigma
,\lambda_{k}}x^{k}$, where the starting point $x^{0}\in\mathcal{H}$ is
arbitrary, the step size function is given by (\ref{e-sigma2}) and the
relaxation parameter $\lambda_{k}\in\lbrack\varepsilon,2-\varepsilon]$ for an
arbitrary constant $\varepsilon\in(0,1)$. Since $U$ is nonexpansive, we would
rather prefer to apply the step size $\sigma(x)$ given by (\ref{e-sigma2a}),
and, as explained in Remark \ref{r-s1/2}, the convergence also holds in this
case. Note that property (\ref{e-ts-min}) does not hold in general for a
system of linear inequalities. An application of the step size $\sigma(x)$
given by (\ref{e-sigma2}) ensures that $U_{\sigma}$ is a cutter but does not
guarantee that $\Vert U_{\sigma}x-z\Vert\leq\Vert Ux-z\Vert$ for any
$z\in\bigcap_{i=1}^{m}C_{i}$, unless $\sigma(x)\geq1$.

The step size $\sigma(x)$ can be also presented, equivalently, in the
following way. Denote, as before, $u^{i}=S_{i}x$, i.e., $u^{0}=x$,
$u^{i}=U_{i}u^{i-1}$, $i=1,2,\ldots,m$. Of course, $u^{m}=Ux$ thus,%
\begin{equation}
Ux=u^{0}+\sum_{i=1}^{m}(u^{i}-u^{i-1})=x-\sum_{i=1}^{m}\frac{(\langle
a^{i},u^{i-1}\rangle-b_{i})_{+}}{\Vert a^{i}\Vert^{2}}a^{i}\text{.}%
\end{equation}
For any $z\in C$ we have,%
\begin{align}
\langle Ux-x,z-x\rangle &  =-\left\langle \sum_{i=1}^{m}\frac{(\langle
a^{i},u^{i-1}\rangle-\beta_{i})_{+}}{\Vert a^{i}\Vert^{2}}a^{i}%
,z-x\right\rangle \nonumber\\
&  =\sum_{i=1}^{m}(\langle a^{i},x\rangle-\langle a^{i},z\rangle
)\frac{(\langle a^{i},u^{i-1}\rangle-\beta_{i})_{+}}{\Vert a^{i}\Vert^{2}%
}\nonumber\\
&  \geq\sum_{i=1}^{m}(\langle a^{i},x\rangle-b_{i})\frac{(\langle
a^{i},u^{i-1}\rangle-b_{i})_{+}}{\Vert a^{i}\Vert^{2}}\text{,}%
\end{align}
i.e.,%
\begin{equation}
\langle Ux-x,z-x\rangle\geq\sum_{i=1}^{m}(\langle a^{i},x\rangle-b_{i}%
)\frac{(\langle a^{i},u^{i-1}\rangle-b_{i})_{+}}{\Vert a^{i}\Vert^{2}}\text{.}%
\end{equation}

The same inequality can be obtained by an application of the first inequality
in (\ref{e-main2}). We can apply the above inequality to define the following
step size%
\begin{equation}
\sigma_{\max}(x)=\frac{\displaystyle\sum_{i=1}^{m}(\langle a^{i}%
,x\rangle-\beta_{i})\frac{\displaystyle(\langle a^{i},u^{i-1}\rangle
-b_{i})_{+}}{\displaystyle\Vert a^{i}\Vert^{2}}}{\Vert Ux-x\Vert^{2}}\text{,}
\label{e-sigma5}%
\end{equation}
where $x\notin\operatorname*{Fix}U$. As mentioned in Remark \ref{r-s-max}, an
application of the following step size%
\begin{equation}
\sigma(x)=\max\left(  \frac{m+1}{2m},\sigma_{\max}(x)\right)  \text{,}
\label{e-sigma(x)}%
\end{equation}
where $x\notin\operatorname*{Fix}U$ (compare with Remark \ref{r-s-max}), seems reasonable.

\begin{corollary}
\label{c-acc-ineq}Let $U:=U_{m}U_{m-1}\cdots U_{1}$, where $U_{i}$ is given by
$\mathrm{(\ref{e-Uix2})}$, $i=1,2,\ldots,m$, the sequence $\{x^{k}%
\}_{k=0}^{\infty}\subset\mathcal{H}$ be defined by the recurrence%
\begin{equation}
x^{k+1}=x^{k}+\lambda_{k}\sigma(x^{k})(Ux^{k}-x^{k}),
\end{equation}

for all $x^{k}\notin\operatorname*{Fix}U$, where $x^{0}\in\mathcal{H}$ is
arbitrary, $\sigma:=\sigma_{\max}$ is given by $\mathrm{(\ref{e-sigma5})}$ and
$\lambda_{k}\in\lbrack\varepsilon,2-\varepsilon]$ for an arbitrary constant
$\varepsilon\in(0,1)$. Then $\{x^{k}\}_{k=0}^{\infty}$ converges weakly to a
solution of the system $\mathrm{(\ref{e-li})}$.
\end{corollary}

One can prove even the strong convegence of $\{x^{k}\}_{k=0}^{\infty}$ in
Corollary \ref{c-acc-ineq}. The proof will be presented elsewhere. As
mentioned in Remark \ref{r-s-max}, Corollary \ref{c-acc-ineq} remains true, is
we set $\sigma:=\max((m+1)/(2m),\sigma_{\max})$.

\subsection{Local acceleration of the sequential cyclic subgradient projection
method\label{subsect:seq-csp}}

Let $c_{i}:\mathcal{H}\rightarrow\mathbb{R}$ be continuous and convex
functions, $i=1,2,\ldots,m$ and consider the following system of convex
inequalities%
\begin{equation}
c_{i}(x)\leq0\text{, }i=1,2,\ldots,m\text{.}%
\end{equation}
Denote $C_{i}:=\{x\in\mathcal{H}\mid c_{i}(x)\leq0\}$, $i=1,2,\ldots,m$, and
assume that $C=\bigcap_{i=1}^{m}C_{i}\neq\emptyset$. Define the operators
$U_{i}:\mathcal{H}\rightarrow\mathcal{H}$ by (\ref{e-Uix3}), $i=1,2,\ldots,m$.
Letting $S_{0}:=I$, $S_{i}:=U_{i}U_{i-1}\cdots U_{1}$, $u^{i}=S_{i}x$ and
$y^{i}=u^{i}-u^{i-1}$ for all $i=1,2,\ldots,m$, we have, by Remark
\ref{r-sum-yi},%
\begin{align}
\sum_{i=1}^{m}\langle S_{m}x-S_{i-1}x,S_{i}x-S_{i-1}x\rangle &  =\sum
_{i=1}^{m}\langle S_{i}x-x,S_{i}x-S_{i-1}x\rangle\nonumber\\
&  =\sum_{i=1}^{m}\langle U_{i}u^{i-1}-x,U_{i}u^{i-1}-u^{i-1}\rangle
\nonumber\\
&  =-\sum_{i=1}^{m}\frac{(c_{i}(u^{i-1}))_{+}}{\Vert g_{i}(u^{i-1})\Vert^{2}%
}\langle U_{i}u^{i-1}-x,g_{i}(u^{i-1})\rangle\text{,}%
\end{align}
and the step size given by (\ref{e-sigma1}) can be written in the form%
\begin{equation}
\sigma_{\max}(x)=-\frac{\displaystyle\sum_{i=1}^{m}\frac{\displaystyle(c_{i}%
(u^{i-1}))_{+}}{\displaystyle\Vert g_{i}(u^{i-1})\Vert^{2}}\langle
U_{i}u^{i-1}-x,g_{i}(u^{i-1})\rangle}{\Vert Ux-x\Vert^{2}}\text{,}
\label{e-sigma3}%
\end{equation}
where $x\notin\operatorname*{Fix}U$.

Let $U_{\sigma,\lambda}$ be defined by (\ref{e-Tx}), where $\lambda\in(0,2)$
and $\sigma(x)$ is given by (\ref{e-sigma3}). We obtain the following corollary.

\begin{corollary}
\label{c-acc-sp}Let $U:=U_{m}U_{m-1}\cdots U_{1}$, where $U_{i}$ is given by
$\mathrm{(\ref{e-Uix3})}$, $i=1,2,\ldots,m$. Let the sequence $\{x^{k}%
\}_{k=0}^{\infty}\subset\mathcal{H}$ be defined by the recurrence%
\begin{equation}
x^{k+1}=U_{\sigma,\lambda_{k}}x^{k}=x^{k}+\lambda_{k}\sigma(x^{k})\frac
{Ux^{k}-x^{k}}{\Vert Ux^{k}-x^{k}\Vert^{2}},
\end{equation}

for all $x^{k}\notin\operatorname*{Fix}U$, where $x^{0}\in\mathcal{H}$ is
arbitrary, $\sigma:=\sigma_{\max}$ is given by $\mathrm{(\ref{e-sigma3})}$,
and $\lambda_{k}\in\lbrack\varepsilon,2-\varepsilon]$ for an arbitrary
constant $\varepsilon\in(0,1)$. Then $\{x^{k}\}_{k=0}^{\infty}\ $converges
weakly to an element of $C$.
\end{corollary}

As mentioned in Remark \ref{r-s-max}, Corollary \ref{c-acc-sp} remains true,
is we set $\sigma:=\max((m+1)/(2m),\sigma_{\max})$.\medskip

\textbf{Acknowledgments}. We thank an anonymous referee for his constructive
comments which helped to improve the paper. This work was partially supported
by United States-Israel Binational Science Foundation (BSF) Grant number
200912 and by US Department of Army Award number W81XWH-10-1-0170.


\medskip

\begin{thebibliography}{99999}                                                                                            %


\bibitem[AS05]{AS05}G. Appleby and D. C. Smolarski, A linear acceleration row
action method for projecting onto subspaces, \textit{Electronic Transactions
on Numerical Analysis}, \textbf{20} (2005), 523--275.

\bibitem[BB96]{BB96}H. H. Bauschke and J. M. Borwein, On projection algorithms
for solving convex feasibility problems, \textit{SIAM Review}, \textbf{38}
(1996), 367--426.

\bibitem[BC01]{BC01}H. H. Bauschke and P. L. Combettes, A weak to strong
convergence principle for Fej\'{e}r-monotone methods in Hilbert spaces,
\textit{Mathematics of Operation Research}, \textbf{26} (2001), 248--264.

\bibitem[Ceg10]{Ceg10}A. Cegielski, Generalized relaxations of nonexpansive
operators and convex feasibility problems, \textit{Contemporary Mathematics,}
\textbf{513} (2010) 111--123.

\bibitem[CC11]{CC11}A. Cegielski and Y. Censor, Opial-type theorems and the
common fixed point problem, in: H. H. Bauschke, R. S. Burachik, P. L.
Combettes, V. Elser, D. R. Luke and H. Wolkowicz (Editors),
\textit{Fixed-Point Algorithms for Inverse Problems in Science and
Engineering}, Springer Optimization and Its Applications 49, New York, NY,
USA, 2011, 155--183.

\bibitem[CS08a]{CS08a}Y. Censor and A. Segal, The split common fixed point
problem for directed operators, \textit{Journal of Convex Analysis},
\textbf{16} (2009), 587--600.

\bibitem[CS09]{CS09}Y. Censor and A. Segal, Sparse string-averaging and split
common fixed points, \textit{Contemporary Mathematics}, \textbf{513 }(2010), 125--142.

\bibitem[CS08]{CS08}Y. Censor and A. Segal, On the string averaging method for
sparse common fixed point problems, \textit{International Transactions in
Operational Research}, \textbf{16 }(2009), 481--494.

\bibitem[Cim38]{Cim38}G. Cimmino, Calcolo approssimato per le soluzioni dei
sistemi di equazioni lineari, \textit{La Ricerca Scientifica XVI Series II,
Anno IX}, \textbf{1 (}1938)\textbf{, }326--333.

\bibitem[Com01]{Com01}P. L. Combettes, Quasi-Fej\'{e}rian analysis of some
optimization algorithms. In: \textit{Inherently Parallel Algorithms in
Feasibility and Optimization and their Applications}, D. Butnariu, Y. Censor
and S. Reich (Editors) Elsevier Science Publishers, Amsterdam, The
Netherlands, 2001, 115--152.

\bibitem[Cro05]{Cro05}G. Crombez, A geometrical look at iterative methods for
operators with fixed points, \textit{Numerical Functional Analysis and
Optimization},\textit{ }\textbf{26} (2005), 157--175.

\bibitem[DPi81]{DPi81}A. R. De Pierro, \textit{Metodos de proje\c{c}\~{a}o
para a resolu\c{c}\~{a}o de sistemas gerais de equa\c{c}\~{o}es alg\'{e}bricas
lienaers}, Tese de Doutoramento, Instituto de Matem\'{a}tica, Universidade
Federal do Rio de Janeiro (IM-UFRJ), 1981.

\bibitem[DSa87]{DSa87}L. T. Dos Santos, A parallel subgradient projections
method for the convex feasibility problem, \textit{Journal of Computational
and Applied Mathematics} \textbf{18} (1987), 307--320.

\bibitem[GPR67]{GPR67}L. G. Gurin, B. T. Polyak and E. V. Raik, The method of
projection for finding the common point in convex sets, \textit{Zh. Vychisl.
Mat. i Mat. Fiz}., \textbf{7} (1967), 1211--1228 (in Russian). English
translation in \textit{USSR Computational Mathematics and Mathematical
Physics,} \textbf{7} (1967), 1--24.

\bibitem[GR84]{GR84}K. Goebel and S. Reich, \textit{Uniform Convexity,
Hyperbolic Geometry, and Nonexpansive Mappings}, Marcel Dekker, New York and
Basel, 1984.

\bibitem[Kac37]{Kac37}S. Kaczmarz, Angen\"{a}herte Aufl\"{o}sung von Systemen
linearer Gleichungen, \textit{Bulletin International de l'Acad\'{e}mie
Polonaise des Sciences et des Lettres}, \textbf{A35} (1937), 355--357.

\bibitem[Mai10]{Mai10}P.-E. Maing\'{e}, The viscosity approximation process
for quasi-nonexpansive mappings in Hilbert spaces, \textit{Computers and
Mathematics with Applications}, \textbf{59} (2010), 74-79.

\bibitem[Opi67]{Opi67}Z. Opial, Weak convergence of the sequence of successive
approximations for nonexpansive mappings, \textit{Bulletin of the American
Mathematical Society}, \textbf{73} (1967), 591--597.

\bibitem[Seg08]{Seg08}A. Segal, \textit{Directed Operators for Common Fixed
Point Problems and Convex Programming Problems}, Ph.D. Thesis, University of
Haifa, Haifa, Israel, September 2008.

\bibitem[YO04]{YO04}I. Yamada and N. Ogura, Hybrid steepest descent method for
variational inequality problem over the fixed point set of certain
quasi-nonexpansive mapping, \textit{Numer. Funct. Anal. and Optimiz.},
\textbf{25} (2004), 619--655.

\bibitem[Zak03]{Zak03}M. Zaknoon, \textit{Algorithmic Developments for the
Convex Feasibility Problem}, Ph.D. Thesis, University of Haifa, Haifa, Israel,
April 2003.

\bibitem[Zar71]{Zar71}E. H. Zarantonello, Projections on convex sets in
Hilbert space and spectral theory. In: \textit{Contributions to Nonlinear
Functional Analysis }(E.H. Zarantonello, Editor), pp. 237--424, Academic
Press, New York, NY, USA, 1971.
\end{thebibliography}
\end{document}